\documentclass{article}

\usepackage{arxiv}

\usepackage[utf8]{inputenc} % allow utf-8 input
\usepackage[T1]{fontenc}    % use 8-bit T1 fonts
\usepackage{hyperref}       % hyperlinks
\usepackage{url}            % simple URL typesetting
\usepackage{booktabs}       % professional-quality tables
\usepackage{amsmath}
\usepackage{amssymb}
\usepackage{microtype}      % microtypography
\usepackage{graphicx}
\usepackage{natbib}
\usepackage{bm}
\usepackage{enumitem}

\title{Asymptotic theory and first-order bias of the Wallace--Freeman estimator}

\author{ \href{https://orcid.org/0000-0003-3017-0871}{\includegraphics[scale=0.06]{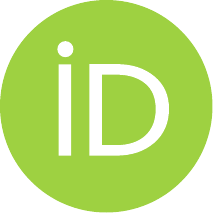}\hspace{1mm}Enes Makalic}%\thanks{Use footnote for providing further
    \\
    Faculty of Information Technology\\
	Monash University\\
	Clayton, VIC 3800 \\
	\texttt{enes.makalic@monash.edu} \\
	\And
	\href{https://orcid.org/0000-0002-1788-2375}{\includegraphics[scale=0.06]{orcid.pdf}\hspace{1mm}Daniel F.~Schmidt} \\
    Faculty of Information Technology\\
	Monash University\\
	Clayton, VIC 3800 \\
	\texttt{daniel.schmidt@monash.edu} \\
}

\DeclareMathOperator*{\argmin}{arg\,min}
\DeclareMathOperator*{\argmax}{arg\,max}

\newcommand{\btheta}{\bm{\theta}}
\newcommand{\bvartheta}{\bm{\vartheta}}

\newcommand{\bu}{\mathbf{u}}
\newcommand{\bX}{\mathbf{X}}
\newcommand{\bx}{\mathbf{x}}

\newtheorem{theorem}{Theorem}[section]

\newtheorem{corollary}[theorem]{Corollary}

\newtheorem{remark}[theorem]{Remark}
\newtheorem{proposition}[theorem]{Proposition}

\newenvironment{proof}[1][Proof]{\par\noindent\textbf{#1.} }{\hfill$\square$\par}

\newcounter{assump}

\renewcommand{\shorttitle}{Asymptotic theory for the Wallace--Freeman estimator}

\hypersetup{
pdftitle={Asymptotic theory and first-order bias of the Wallace--Freeman estimator},
pdfsubject={Statistics methodology and theory},
pdfauthor={Enes Makalic, Daniel F. Schmidt},
pdfkeywords={Penalised likelihood, M-estimation, Minimum message length, Maximum likelihood, Cox--Snell bias, Asymptotic normality, Weibull distribution},
}

\begin{document}
\maketitle

\begin{abstract}
The Wallace--Freeman estimator is a classical minimum message length estimator whose relationship with likelihood-based asymptotic theory has not been fully developed. We show that, in regular parametric models, the Wallace--Freeman criterion is equivalent, up to constants, to a penalised likelihood criterion with penalty weight \(n^{-1}\). This representation places the estimator within the standard theory of penalised M-estimation and yields existence, consistency, an asymptotic linear expansion, and asymptotic normality under regularity conditions. We further derive the first-order difference between the Wallace--Freeman estimator and the maximum likelihood estimator, showing that it is an explicit \(O(n^{-1})\) shift determined by the gradient of the Wallace--Freeman penalty. Combining this expansion with the Cox--Snell formula gives a first-order bias expansion for the Wallace--Freeman estimator. The result clarifies its relationship with maximum likelihood, Jeffreys-prior penalisation, and Firth-type bias reduction. We illustrate the theory for the Weibull model, where the penalty modifies the leading bias of the maximum likelihood estimator of the shape parameter.
\end{abstract}

% keywords can be removed
\section{Introduction}
\label{sec:intro}
The Wallace--Freeman minimum message length estimator~\cite{WallaceFreeman87,Wallace05} is a classical point estimator derived from the minimum message length principle~\cite{WallaceBoulton68,WallaceBoulton75,WallaceDowe99a,WallaceDowe99c,Wallace05}. Although the estimator has traditionally been motivated by information-theoretic coding arguments, its connection with standard likelihood theory, and in particular with penalised estimation and higher-order bias expansions, has not been systematically developed. This paper shows that the Wallace--Freeman estimator can be written as a penalised M-estimator with penalty weight \(n^{-1}\). This representation provides a direct route to consistency, asymptotic normality, and an explicit \(O(n^{-1})\) bias expansion that extends the Cox--Snell bias formula for maximum likelihood estimators~\cite{CoxSnell68,CordeiroKlein94}.

The Wallace--Freeman criterion, up to constants independent of the parameters $\btheta \in \Theta \subseteq \mathbb{R}^d$, is
\[
-\log p_n(\bx \mid \btheta)
-\log\pi(\btheta)
+\frac12\log |I(\btheta)|,
\]
where $p_n(\bx \mid \btheta)$ denotes the likelihood of $n$ data points $\bx \in \mathcal{X}^n$, $I(\btheta)$ is the per-sample Fisher information matrix, and $\pi(\btheta)$ is a prior density on $\Theta$. Equivalently, we can view the criterion as the empirical risk penalised by
\[
n^{-1}\left\{-\log\pi(\btheta)+\frac12\log |I(\btheta)|\right\},
\]
where the penalty vanishes at rate \(n^{-1}\). Consequently, the Wallace--Freeman estimator has the same first-order limiting distribution as the maximum likelihood estimator, but differs from it at order \(n^{-1}\) by a deterministic term involving the gradient of the Wallace--Freeman penalty. This term yields a simple modification of the Cox--Snell first-order bias expansion.

The contributions of the paper are threefold. First, we derive the Wallace--Freeman codelength~\cite{WallaceFreeman87,Wallace05} from a local high-rate approximation to strict minimum message length~\cite{WallaceBoulton75,Wallace05,MakalicSchmidt26b}. Second, we formulate the resulting estimator as a penalised M-estimator and establish existence, consistency, an asymptotic linear representation, and asymptotic normality under standard regularity conditions. Third, we derive the first-order difference between the Wallace--Freeman and maximum likelihood estimators, yielding an explicit penalty-driven correction to the \(O(n^{-1})\) maximum likelihood bias.

As an illustration of the general theory, we study the Wallace--Freeman estimator for the Weibull model. This example is useful because the maximum likelihood estimator of the Weibull shape parameter has a non-negligible first-order bias. The proposed expansion gives an explicit
Wallace--Freeman correction term, showing how the prior contribution and the Fisher information term modify the Cox--Snell bias. 

The remainder of the paper is organised as follows. Section~\ref{sec:mml} reviews minimum message length inference, strict minimum message length, and the Wallace--Freeman approximation. Section~\ref{sec:wf} derives the Wallace--Freeman codelength from a local high-rate SMML expansion and formulates the estimator as a penalised M-estimator. Section~\ref{sec:M:estimator} establishes its large-sample properties. Section~\ref{sec:bias} derives the
first-order bias expansion and its relation to the Cox--Snell formula. The theory is illustrated for the Weibull distribution in Section~\ref{sec:example}, with discussion and future work given in Section~\ref{sec:discussion}.
\section{Minimum message length and the Wallace--Freeman approximation}
\label{sec:mml}
Minimum message length (MML) \cite{WallaceBoulton68,WallaceBoulton75,WallaceFreeman87,Wallace05} is a coding approach to parametric inference and model selection. The key idea is to cast parameter estimation and model selection in the framework of data compression. Given data \(\bx \in \mathcal X^n\) and  parameters \(\btheta\in\Theta\subset\mathbb R^d\), a  MML message describing the data consists of two parts: (1)~the assertion of length $\mathcal{I}(\btheta)$, which encodes the parameters, and (2)~the detail of length $\mathcal{I}(\bx \mid \btheta)$, which encodes the data conditional on the asserted parameters. Therefore the MML codelength has the form
\begin{align}
\mathcal{I}(\bx,\btheta)
=
\mathcal{I}(\btheta)+\mathcal{I}(\bx \mid \btheta).
\end{align}
This decomposition gives the form of a two-part message, but it does
not by itself specify a finite code. In particular, a complete MML code must determine which parameter assertions are available and which data sets are encoded using each assertion. Strict minimum message length (SMML)~\cite{WallaceBoulton75,Wallace05,MakalicSchmidt26b} makes this construction global by partitioning the sample space. Each cell of the partition is assigned an assertion probability and a representative parameter value, and all data sets falling in that cell are encoded using the same representative. Thus, whereas many practical MML estimators arise by minimising an approximate pointwise codelength over \(\btheta\), SMML is defined as an optimisation over finite codebooks induced by measurable partitions of the sample space. We now describe this construction formally.

Let \(\Pi\) denote a prescribed class of admissible measurable partitions of \(\mathcal X^n\). For \(\mathcal P\in\Pi\), write
$
\mathcal P=\{P_1,\ldots,P_k\},
$
where \(k\) is finite.
Let $\pi(\btheta)$ denote a prior density on $\Theta$, and define the marginal distribution
\begin{align*}
    r_n(\bx)=\int_\Theta p_n(\bx \mid \btheta)\pi(\btheta)\,d\btheta.
\end{align*}
For a cell \(P_j \in \mathcal{P}\), write
\begin{align}
    q_j=r_n(P_j) = \sum_{\bx\in P_j} r_n(\bx) > 0, \qquad (1 \leq j \leq k),
    \label{eqn:qj}
\end{align}
with sums interpreted as integrals in the continuous case.  If data falling in cell \(P_j\) are encoded using the representative assertion \(\btheta_j\), then the expected two-part SMML codelength~\cite{WallaceBoulton75,Wallace05,MakalicSchmidt26b} is
\begin{align}
\mathcal{I}(\mathcal P, \btheta_1, \ldots, \btheta_k)
=
-\sum_{j=1}^k q_j\log q_j
-\sum_{j=1}^k \sum_{\bx\in P_j} r_n(\bx)\log p_n(\bx \mid \btheta_j) .
\label{eqn:smml}
\end{align}
For a fixed partition \(\mathcal P\), the optimal representative in cell \(P_j\) is
\begin{align}
\btheta_j^*
=
\argmax_{\btheta\in\Theta}
\sum_{\bx\in P_j} r_n(\bx)\log p_n(\bx \mid \btheta).
\label{eqn:smml:estimate}
\end{align}
The optimal SMML solution is obtained by minimising
$\mathcal{I}(\mathcal P, \btheta_1, \ldots, \btheta_k)$
over all measurable partitions \(\mathcal P \in \Pi\) of the sample space.
In general, computing the optimal SMML partition is intractable except in special low-dimensional cases \cite{FarrWallace02,Dowty13}. The Wallace--Freeman approximation \cite{WallaceFreeman87} avoids this global partition problem by replacing it with a local high-rate approximation, which we formalise in the next section.

\section{Asymptotic SMML and the Wallace--Freeman estimator}
\label{sec:wf}
Consider an asymptotic coding regime in which the
number of SMML assertion cells is large and each cell has small probability. In this high rate regime, each SMML cell of the data-space partition may be associated locally with a small region in parameter space. The assertion codelength is governed by the prior mass of this region, whereas the increase in detail codelength from using a representative parameter is approximated by local Kullback--Leibler geometry, or equivalently by Fisher information curvature. 
%Under Fisher normalisation, this local regret becomes a quadratic quantisation problem. 
Optimising the resulting cell-volume term yields the Wallace--Freeman criterion~\cite{WallaceFreeman87,Wallace05}.

For completeness, we formalise this derivation for a regular sequence of high-rate SMML codebooks. The assumptions below encode local Fisher quadraticity, prior-mass localisation of cells, negligible representation error, and asymptotic quantisation optimality in Fisher-normalised coordinates.

\begin{theorem}[High-rate expected SMML expansion]
\label{thm:asmml-expected-simplified}
Let $\Theta\subseteq\mathbb R^d$ be open, and let
$P_{\btheta}^{(n)}$ denote the distribution of the sample
\[
\bX=(X_1,\ldots,X_n)
\]
under parameter $\btheta$. Write $p_n(\bx\mid\btheta)$ for the likelihood and let $J_n(\btheta)$ denote the expected Fisher information for the sample of size $n$. Let $K \Subset \Theta$ be compact. Suppose that \(\pi\) is a probability density on \(K\),
\[
\int_K \pi(\btheta)\,d\btheta=1,
\]
and that \(\pi\) is continuous and strictly positive on \(K\).
Define the prior predictive density
\[
r_n(\bx)
=
\int_K p_n(\bx\mid\btheta)\pi(\btheta)\,d\btheta .
\]
Consider a sequence of high-rate SMML codebooks $\mathcal C_n=(\mathcal P_n,q_n,\btheta_n)$, where
\[
\mathcal P_n=\{P_{n,1},\ldots,P_{n,k_n}\},
\qquad
q_n=\{q_{n,1},\ldots,q_{n,k_n}\},
\qquad
\btheta_n = \{ \btheta_{n,1},\ldots, \btheta_{n,k_n}\},
\]
with optimal assertion probabilities given by
\[
q_{n,j}=r_n(P_{n,j}) > 0, \qquad \forall n,j.
\]
For data $\bx\in P_{n,j}$, let $T_n(\bx)\in K$ denote the corresponding (unquantised) local parameter estimate, and let $\btheta_{n,j}\in K$ denote the quantised assertion used for all data in cell $P_{n,j} \in \mathcal P_n$. 

Assume the following conditions hold.

\begin{enumerate}[label=(A\arabic*)]

\item \label{ass:simplified:fisher}
%\textbf{Regular local Fisher geometry.}
The $n$-sample Fisher information $J_n(\btheta)$ is continuous and positive definite on $K$, and
\[
\inf_{\btheta\in K}
\lambda_{\min}\{J_n(\btheta)\}
\rightarrow \infty .
\]

\item \label{ass:simplified:kl}
For every bounded set $B\subset\mathbb R^d$,
\[
\sup_{\btheta\in K,\ \bu\in B}
\left|
D_{\mathrm{KL}}\!\left(
P_{\btheta+J_n(\btheta)^{-1/2}\bu}^{(n)}
\Vert
P_{\btheta}^{(n)}
\right)
-
\frac12\|\bu\|^2
\right|
\to 0,
\]
where $D_{\mathrm{KL}}(\cdot\Vert\cdot)$ denotes the Kullback--Leibler divergence~\cite{KullbackLeibler51}. 
%For each fixed bounded $B$, the point $\btheta+J_n(\btheta)^{-1/2}\bu$ belongs to $\Theta$ for all sufficiently large $n$, uniformly over $\btheta\in K$ and $\bu\in B$.

\item \label{ass:simplified:localisation}
Each data-space cell \(P_{n,j} \in \mathcal P_n\) induces a parameter-space region \(W_{n,j}\subset K\), interpreted as the set of unquantised local parameter values \(T_n(\bx)\) assigned to the assertion \(\btheta_{n,j}\). Assume that
\[
q_{n,j}
=
\Pi(W_{n,j})\exp(\alpha_{n,j}),
\qquad
\Pi(W)
=
\int_W\pi(\bvartheta)\,d\bvartheta,
\]
where
\[
\sum_j q_{n,j}|\alpha_{n,j}|\to0.
\]
Moreover,
\[
W_{n,j}
=
\left\{
\btheta_{n,j}
+
J_n(\btheta_{n,j})^{-1/2}\bu : \bu \in C_{n,j}
\right\},
\]
where \(C_{n,j}\subset\mathbb R^d\) is measurable, contains the origin, has positive volume
\[
V_{n,j}=\operatorname{Vol}(C_{n,j}),
\]
and the sets \(C_{n,j}\) have uniformly bounded diameter.

\item \label{ass:simplified:regret}
%\textbf{KL representation of the second-part regret.}
Let
\[
\Delta_{n,j}
=
\mathbb E_{r_n}
\left[
-\log p_n(\bX\mid\btheta_{n,j})
+
\log p_n\{\bX\mid T_n(\bX)\}
\,\middle|\,
\bX\in P_{n,j}
\right].
\]
Assume that
\[
\Delta_{n,j}
=
\frac{1}{\Pi(W_{n,j})}
\int_{W_{n,j}}
D_{\mathrm{KL}}\!\left(
P_{\bvartheta}^{(n)}
\Vert
P_{\btheta_{n,j}}^{(n)}
\right)
\pi(\bvartheta)\,d\bvartheta
+
\varepsilon_{n,j},
\]
where
\[
\sum_j q_{n,j}|\varepsilon_{n,j}|\longrightarrow 0 .
\]
%
%This assumption identifies the conditional increase in detail codelength with the prior-weighted KL regret of replacing the unquantised local parameter by the cell representative.

\item \label{ass:simplified:quantisation}
%\textbf{Asymptotic high-rate quantisation optimality.}
For a Fisher-normalised cell $C$, define $V(C) = {\rm Vol}(C)$,
\[
\Phi(C)
=
-\log V(C)
+
\frac{1}{2V(C)}
\int_C \|\mathbf u\|^2\,d\mathbf u ,
\]
and its scale-invariant normalised second moment by
\[
\kappa(C)
=
\frac{1}{d V(C)^{1+2/d}}
\int_C \|\mathbf u\|^2\,d\mathbf u .
\]
For $a>0$, let $a C=\{a\mathbf u:\mathbf u\in C\}$. Since $\kappa(aC)=\kappa(C)$, minimising $\Phi(aC)$ over $a>0$ gives
\[
\inf_{a>0}\Phi(aC)
=
\frac d2\log \kappa(C)
+
\frac d2 .
\]
Let
\[
\kappa_d
=
\inf_C \kappa(C)
\]
be the optimal $d$-dimensional normalised second-moment constant~\cite{ConwaySloane98}. We assume that the regular high-rate cell sequence is asymptotically quantisation-optimal in Fisher-normalised coordinates, namely
\[
\sum_j q_{n,j}\Phi(C_{n,j})
=
\frac d2\log\kappa_d
+
\frac d2
+
o(1).
\]

\item \label{ass:simplified:empirical}
%\textbf{Prior-weighted empirical convergence of assertions.}
With
\[
\varphi_n(\btheta)
=
-\log\pi(\btheta)
+
\frac12\log|J_n(\btheta)|,
\]
assume that
\[
\sum_j q_{n,j}\varphi_n(\btheta_{n,j})
=
\int_K \varphi_n(\btheta)\pi(\btheta)\,d\btheta
+
o(1).
\]

\end{enumerate}

Let $\mathcal I_n^{\mathrm{ASMML}}$ denote the expected codelength of a high-rate SMML sequence satisfying
Assumptions~\ref{ass:simplified:fisher}--\ref{ass:simplified:empirical}.
Then
\begin{align}
\mathcal I_n^{\mathrm{ASMML}}
&=
\mathbb E_{r_n}
\left[
-\log p_n\{\bX\mid T_n(\bX)\}
\right]
+
h(\pi)
+
\frac12\int_K \pi(\btheta)\log|J_n(\btheta)|\,d\btheta
+
\frac d2\log\kappa_d
+
\frac d2
+
o(1),
\label{eq:asmml-expected-simplified}
\end{align}
where
\[
h(\pi)
=
-\int_K \pi(\btheta)\log\pi(\btheta)\,d\btheta .
\]

In particular, if the model is i.i.d. and $J_n(\btheta)=n I(\btheta)$, then
\begin{align}
\mathcal I_n^{\mathrm{ASMML}}
&=
\mathbb E_{r_n}
\left[
-\log p_n\{\bX\mid T_n(\bX)\}
\right]
+
\frac d2\log n
+
h(\pi)
+
\frac12\int_K \pi(\btheta)\log|I(\btheta)|\,d\btheta
+
\frac d2\log\kappa_d
+
\frac d2
+
o(1).
\label{eq:asmml-expected-iid-simplified}
\end{align}
\end{theorem}

\begin{remark}[Role of the high-rate assumptions]
\label{rem:high-rate-assumptions}
Assumptions~\ref{ass:simplified:fisher}--\ref{ass:simplified:kl} impose the local Fisher geometry where Fisher-normalised neighbourhoods shrink in parameter space and the local Kullback--Leibler divergence is asymptotically quadratic. Assumptions~\ref{ass:simplified:localisation}--\ref{ass:simplified:regret} then connect the coding partition to this local statistical geometry by requiring each data-space cell to correspond, up to negligible error, to a small parameter-space region whose assertion probability is its prior mass and whose excess detail codelength is its average KL regret.

Assumption~\ref{ass:simplified:quantisation} is the high-rate quantisation condition that, after Fisher normalisation, reduces the local coding problem to balancing cell volume against squared-distance regret. Assumption~\ref{ass:simplified:empirical} is an averaging condition ensuring that the representative assertions are distributed according to the prior measure at the level needed for the expected codelength expansion. Together, these conditions formalise the Wallace--Freeman heuristic that assertion costs are controlled by prior mass, detail costs by Fisher information curvature, and optimisation over local cell volume yields the Wallace--Freeman penalty.
\end{remark}

\begin{proposition}[Wallace--Freeman codelength]
\label{cor:wf-from-asmml-expected}
Assume the conditions of Theorem~\ref{thm:asmml-expected-simplified}. Suppose further that the expected high-rate expansion admits the following local, pointwise realisation. Uniformly for \(\btheta\in K\), and locally
uniformly for Fisher-normalised cell volumes \(V>0\),
%
%Assume the conditions of Theorem~\ref{thm:asmml-expected-simplified}. Suppose further that the high-rate expansion is realised locally, uniformly for $\btheta\in K$, by Fisher-normalised cells of volume $V>0$. Specifically,
%assume that, for $V$ in compact subsets of $(0,\infty)$,
%
\begin{align} 
\Lambda_n^{\mathrm{ASMML}}(\bx,\btheta;V)
&=
-\log p_n(\bx\mid\btheta)
-\log\pi(\btheta)
+\frac12\log|J_n(\btheta)|
-\log V
+
\frac{d\kappa_d}{2}V^{2/d}
+
o_{P_{\btheta}^{(n)}}(1),
\label{eq:asmml-local-from-expected}
\end{align}
where the remainder is uniform in $\btheta\in K$ and locally uniform in $V$.
Then the leading volume-dependent term
\[
g(V)
=
-\log V+\frac{d\kappa_d}{2}V^{2/d}
\]
is uniquely minimised at
\[
V_{\mathrm{opt}}
=
\kappa_d^{-d/2}.
\]
At this optimum,
\[
g(V_{\mathrm{opt}})
=
\frac d2\log\kappa_d+\frac d2.
\]
Hence the locally optimised asymptotic SMML codelength is
\begin{align}
\Lambda_n^{\mathrm{ASMML}}(\bx,\btheta;V_{\mathrm{opt}})
&=
-\log p_n(\bx\mid\btheta)
-\log\pi(\btheta)
+\frac12\log|J_n(\btheta)|
+\frac d2\log\kappa_d
+\frac d2
+
o_{P_{\btheta}^{(n)}}(1).
\end{align}

Consequently, the Wallace--Freeman codelength is
\begin{align}
\mathcal I_{\mathrm{WF}}(\bx,\btheta)
&=
-\log p_n(\bx\mid\btheta)
-\log\pi(\btheta)
+\frac12\log|J_n(\btheta)|
+\frac d2\log\kappa_d
+\frac d2 .
\label{eq:wf-codelength}
\end{align}
The corresponding local Wallace--Freeman estimator is
\begin{align}
\hat{\btheta}_{\mathrm{WF}}(\bx)
=
\argmin_{\btheta\in K}
\mathcal I_{\mathrm{WF}}(\bx,\btheta).
\label{eqn:wf-estimator-local}
\end{align}
Equivalently, since the final two terms in
\eqref{eq:wf-codelength} are independent of $\btheta$,
\[
\hat{\btheta}_{\mathrm{WF}}(\bx)
=
\argmin_{\btheta\in K}
\left\{
-\log p_n(\bx\mid\btheta)
-\log\pi(\btheta)
+
\frac12\log|J_n(\btheta)|
\right\}.
\]
If this minimiser lies in the interior of $K$, and the local expansion holds on a compact neighbourhood of the minimiser, then $\hat{\btheta}_{\mathrm{WF}}$ coincides with the usual Wallace--Freeman estimator~\cite{WallaceFreeman87}.
\end{proposition}

Theorem~\ref{thm:asmml-expected-simplified} and Proposition~\ref{cor:wf-from-asmml-expected} detail the derivation of the Wallace--Freeman penalty. The remaining sections are concerned with the statistical properties of the resulting estimator. For this purpose, the high-rate quantisation constants are irrelevant as they do not depend on \(\btheta\). In i.i.d. models, after normalising the codelength by \(n\), the Wallace--Freeman criterion is an empirical likelihood criterion penalised by a term with weight \(n^{-1}\). This places the estimator directly within the standard theory of penalised M-estimation.

%The Wallace--Freeman estimator has been used in a wide range of statistical models, including factor analysis~\cite{WallaceFreeman92}, decision trees~\cite{WallacePatrick93}, and mixture models~\cite{WallaceDowe00}.

%
%
\section{The Wallace--Freeman estimator as a penalised M-estimator}
\label{sec:M:estimator}
For the remainder of the paper we specialise to the i.i.d. case, so that
\[
p_n(\bx\mid\btheta)=\prod_{i=1}^n p(\bx_i\mid\btheta),
\qquad
J_n(\btheta)=nI(\btheta),
\]
where $I(\btheta)$ denotes the per-observation Fisher information.
The Wallace--Freeman estimator~\eqref{eqn:wf-estimator-local} can be written, up to $\btheta$-independent constants, as the minimiser of a penalised empirical criterion of the form
\begin{align}
Q_n(\btheta)
=
\frac1n \sum_{i=1}^n \rho(\bx_i,\btheta)
+
\lambda_n \operatorname{pen}(\btheta),
\label{eq:generic-penM}
\end{align}
where
\begin{align}
\rho(\bx_i,\btheta) = -\log p(\bx_i \mid \btheta),
\quad
\operatorname{pen}(\btheta) = -\log \pi(\btheta) + \frac12 \log |I(\btheta)|,
\label{eq:wf-penalty}
\end{align}
and $\lambda_n = 1/n$. Thus,
\begin{align}
\hat{\btheta}_{\mathrm{WF}}(\bx)
:=
\hat{\btheta}_n
=
\argmin_{\btheta\in\Theta}
\left\{
\frac1n \sum_{i=1}^n \bigl[-\log p(\bx_i \mid \btheta)\bigr]
+
\frac1n\left(
-\log \pi(\btheta)+\frac12\log |I(\btheta)|
\right)
\right\}.
\label{eq:wf87:penalised:m}
\end{align}
The prior contribution favours regions of larger prior mass, while the Fisher information term penalises regions of high local curvature. The Fisher information term is also what ensures invariance under smooth reparameterisation. Since the penalty enters at order $n^{-1}$, the Wallace--Freeman estimator is asymptotically first-order equivalent to maximum likelihood while differing at smaller order through an explicit deterministic correction.

Let
%
%\begin{align}    
$\psi(\bx,\btheta)=\nabla_\theta \rho(\bx,\btheta)$,
%\end{align}
%
and define the stationarity equation
\begin{align}
\Psi_n(\btheta)
=
\frac1n\sum_{i=1}^n \psi(\bx_i,\btheta)
+
\lambda_n \nabla_\theta \operatorname{pen}(\btheta).
\label{eq:wf-score}
\end{align}
The next theorem gives standard large-sample properties of the Wallace--Freeman estimator under regularity conditions. Its proof is deferred to Appendix~\ref{app:penm-proof}.
\begin{theorem}[Large-sample behaviour of the Wallace--Freeman estimator]
\label{thm:wf87:penm}
Assume that \(X_1,\dots,X_n\) are i.i.d. under \(p(\cdot \mid\btheta_0)\) and
\begin{enumerate}[label=(B\arabic*)]
\item The function $Q(\bm{\theta}) := \mathbb{E}_{\theta_0} \left\{ \rho(X, \bm{\theta}) \right\}$ has a unique minimiser at $\bm{\theta}_0 \in {\rm int}(\Theta)$.
\item $\rho(x,\bm{\theta})$ is twice continuously differentiable in a neighbourhood of $\bm{\theta}_0$ and
\begin{align*}
    {\bf A} :=\nabla^2_\theta Q(\bm{\theta}_0) = \mathbb{E}_{\theta_0} \left\{ \nabla_\theta \psi(X, \bm{\theta}_0)\right\}
\end{align*}
exists  and is positive definite, where $\psi({\bf x},\bm{\theta}) = \nabla_\theta \rho({\bf x},\bm{\theta})$.
\item For every compact set $K \subset \Theta$,
\begin{align}
    \sup_{\bm{\theta} \in K} \left| \frac{1}{n}\sum_{i=1}^n \rho({\bf x}_i, \bm{\theta}) - Q(\bm{\theta})\right| \xrightarrow[]{p} 0 .
\end{align}
\item $\lambda_n \to 0$ and $\sqrt{n} \lambda_n \to 0$. In particular, $\lambda_n = 1/n$ satisfies this.
\item ${\rm pen}(\bm{\theta})$ is twice continuously differentiable in a neighbourhood of $\bm{\theta}_0$ and $\nabla_\theta {\rm pen}(\bm{\theta})$ is locally bounded near $\bm{\theta}_0$.
\item Either (i)~ $\Theta$ is compact, or (ii)~$\Theta$ is closed and $Q_n(\bm{\theta})$ is coercive with probability approaching 1, ensuring the minimum is attained.
\item 
\begin{align}
\frac{1}{\sqrt n}\sum_{i=1}^n \psi(X_i,\bm{\theta}_0)
\xrightarrow{d} N({\bf 0}, {\bf B}),
\qquad
{\bf B}:=\mathbb{E}_{\theta_0}\{\psi(X,\bm{\theta}_0)\psi(X,\bm{\theta}_0)'\}.
\end{align}
\item For some neighbourhood \(U\) of \(\bm{\theta}_0\),
\begin{align}
\sup_{\bm{\theta}\in U}
\left\|
\frac1n\sum_{i=1}^n \nabla_\theta\psi(X_i,\bm{\theta})
- \mathbb{E}_{\theta_0}\{\nabla_\theta\psi(X,\bm{\theta})\}
\right\|
\xrightarrow{p}0.
\end{align}
\end{enumerate}
%
% The stationarity condition for the Wallace--Freeman estimator is
% %
% \begin{align}
%     \Psi_n(\bm{\theta}) = \frac{1}{n} \sum_{i=1}^n \psi({\bf x}_i, \bm{\theta}) + \lambda_n \nabla_{\bm{\theta}} {\rm pen}(\bm{\theta}) = 0,
% \end{align}
% %
% with $\lambda_n = 1/n$.
%
\begin{enumerate}[label=(\roman*)]
\item Existence: A global minimiser $\hat{\bm{\theta}}_n \in \Theta$ exists with probability tending to 1.

\item Consistency: $\hat{\btheta}_n\xrightarrow[]{p} \btheta_0$.

\item Rate of convergence:
\[
\hat{\btheta}_n-\btheta_0
=
-\mathbf{A}^{-1}
\frac1n\sum_{i=1}^n\psi(X_i,\btheta_0)
+
o_p(n^{-1/2}),
\]
and hence
\[
\hat{\btheta}_n-\btheta_0=O_p(n^{-1/2}).
\]
%The \(O(n^{-1})\) displacement induced by the Wallace--Freeman penalty is not resolved by this first-order expansion; it is derived in Theorem~\ref{thm:wf87:bias}.
%
%
\item Asymptotic distribution:
\begin{align}
\sqrt{n} (\hat{\bm{\theta}}_n - \bm{\theta}_0) \xrightarrow[]{d} {\rm N}({\bf 0}_d, {\bf A}^{-1} {\bf B} {\bf A}^{-1}), 
\quad 
{\bf B} := \mathbb{E}_{\theta_0} \left\{ \psi(X,\bm{\theta}_0) \psi(X,\bm{\theta}_0)^\prime \right\} .
\end{align}
Moreover, if $\rho({\bf x},\bm{\theta}) = -\log p({\bf x} \mid \bm{\theta})$ is correctly specified so that ${\bf A} = {\bf B} = {\bf I}(\bm{\theta}_0)$, then
\begin{align*}
    \sqrt{n} (\hat{\bm{\theta}}_n - \bm{\theta}_0) \xrightarrow[]{d} {\rm N}({\bf 0}_d, {\bf I}(\bm{\theta}_0)^{-1} ).
\end{align*}
\end{enumerate}
\end{theorem}
\begin{proof}[Proof sketch]
The result follows from standard penalised M-estimation arguments. Under Assumptions (B1)--(B8) the criterion $Q_n$ admits a global minimiser, converges uniformly on compacts to its population counterpart, and the penalty vanishes asymptotically. A first-order Taylor expansion of the estimating equation~\eqref{eq:wf-score} around $\btheta_0$, together with the central limit theorem and the local law of large numbers for the derivative process, yields the asymptotic linear expansion and limiting distribution. Details are given in Appendix~\ref{app:penm-proof}.
\end{proof}
Evaluated at its minimiser, the Wallace--Freeman codelength has the form
\[
\mathcal I_{\mathrm{WF}}(\bx,\hat\btheta_{\mathrm{WF}})
=
-\log p_n(\bx\mid\hat\btheta_{\mathrm{WF}})
+
\frac d2\log n
+
O_p(1).
\]
Since \(\hat\btheta_{\mathrm{WF}}\) and \(\hat\btheta_{\mathrm{MLE}}\) are first-order equivalent, the above codelength has the same leading large-sample form as the well-known Bayesian information criterion (BIC)~\cite{Schwarz78}, with the prior, Fisher information, and quantisation terms contributing only at order \(O(1)\).
\section{Bias of the Wallace--Freeman estimator}
\label{sec:bias}
We now quantify the first-order difference between the Wallace--Freeman estimator and the maximum likelihood estimator, and the resulting effect on asymptotic bias. Let $\ell(\btheta)=\log p_n(\bx \mid \btheta)$ denote the full-sample log-likelihood, and write
\begin{align}
    \ell_i = \frac{\partial \ell}{\partial \theta_i},
    \qquad
    \ell_{ij} = \frac{\partial^2 \ell}{\partial \theta_i \partial \theta_j},
    \qquad
    \ell_{ijk} = \frac{\partial^3 \ell}{\partial \theta_i \partial \theta_j \partial \theta_k},
\end{align}
with corresponding expectations
\begin{align}
    \kappa_{ij} = \mathbb{E}_{\btheta_0}\{\ell_{ij}\},
    \qquad
    \kappa_{ijk} = \mathbb{E}_{\btheta_0}\{\ell_{ijk}\},
    \qquad
    \kappa_{ij,k}
    =
    \left.
    \frac{\partial \kappa_{ij}}{\partial \theta_k}
    \right|_{\btheta=\btheta_0}.
\end{align}
Let
\[
K:=K(\btheta_0)=\{-\kappa_{ij}\}=nI(\btheta_0),
\qquad
K^{-1}=\{\kappa^{ij}\}.
\]

The classical first-order bias expansion for the maximum likelihood estimator due to Cox and Snell~\cite{CoxSnell68} is
\begin{align}
    [\operatorname{Bias}(\hat{\btheta}_{\rm MLE})]_s
    =
    \sum_{i=1}^d \sum_{j=1}^d \sum_{l=1}^d
    \kappa^{si}\kappa^{jl}
    \left(
        \kappa_{ij,l}-\frac12 \kappa_{ijl}
    \right)
    + O(n^{-2}),
    \qquad
    s=1,\ldots,d,
    \label{eqn:bias:MLE:coxsnell}
\end{align}
or, equivalently~\cite{CordeiroKlein94},
\begin{align}
    \operatorname{Bias}(\hat{\btheta}_{\rm MLE})
    =
    K^{-1}A\,\operatorname{vec}(K^{-1})
    + O(n^{-2}),
    \label{eqn:bias:MLE:ck}
\end{align}
where the $(d\times d^2)$ cumulant matrix $A$ is
\begin{align}
\label{eqn:cumulants}
    A
    =
    \bigl[
    A^{(1)} \mid A^{(2)} \mid \cdots \mid A^{(d)}
    \bigr],
    \qquad
    A^{(l)}=\{a_{ij}^{\,l}\},
    \qquad
    a_{ij}^{\,l}
    =
    \kappa_{ij}^{(l)}-\frac12\kappa_{ijl},
    \qquad
    \kappa_{ij}^{(l)}
    =
    \frac{\partial \kappa_{ij}}{\partial \theta_l},
\end{align}
for $i,j,l=1,\ldots,d$.
The next theorem shows that the Wallace--Freeman estimator differs from the maximum likelihood estimator by an explicit $O(n^{-1})$ shift, and that the Cox--Snell expansion is modified by an additional penalty-driven bias term. The proof is deferred to Appendix~\ref{app:bias-proof}.

\begin{theorem}[Bias of the Wallace--Freeman estimator]
\label{thm:wf87:bias}
Let $\bx\in\mathcal{X}^n$ be generated from the model with true parameter $\btheta_0\in\operatorname{int}(\Theta)$. Assume that:
\begin{enumerate}[label=(C\arabic*)]
%\item $\ell(\btheta)$ is three times continuously differentiable in a neighbourhood of $\btheta_0$, the required third derivatives have finite moments, and differentiation and expectation may be interchanged up to third order;
\item $\ell(\btheta)$ is twice continuously differentiable in a neighbourhood of \(\btheta_0\), \(\hat{\btheta}_{\rm MLE}\) is an interior root-\(n\) consistent solution of \(\nabla\ell(\btheta)=0\), and \(\hat{\btheta}_{\rm MLE}\) satisfies the Cox--Snell expansion~\eqref{eqn:bias:MLE:ck}.

\item the observed information satisfies
\begin{align}
    -\nabla^2 \ell(\hat{\btheta}_{\rm MLE})
    =
    K(\btheta_0)+O_p(n^{1/2}),
\end{align}
with $K(\btheta_0)$ nonsingular;

\item the maximum likelihood estimator is root-$n$ consistent:
\[
\hat{\btheta}_{\rm MLE}-\btheta_0 = O_p(n^{-1/2});
\]

\item the vector
\begin{align}
    {\bf a}(\btheta)
    :=
    -\nabla_\theta \operatorname{pen}(\btheta)
    =
    \nabla_\theta \log \pi(\btheta)
    -
    \frac12 \nabla_\theta \log |I(\btheta)|
    \label{eqn:wfbias:a}
\end{align}
is continuously differentiable in a neighbourhood of $\btheta_0$;
\item the remainder in the first-order Taylor expansion of $\nabla \ell(\hat{\bm{\theta}}_{\rm WF})$ about $\hat{\bm{\theta}}_{\rm MLE}$ is $o_p(n^{-1})$, and the corresponding error term is uniformly integrable after multiplication by $n$.
\end{enumerate}

Then
\begin{align}
\hat{\btheta}_{\rm WF}-\hat{\btheta}_{\rm MLE}
=
K(\btheta_0)^{-1}{\bf a}(\btheta_0)+o_p(n^{-1})
=
\frac1n I(\btheta_0)^{-1}{\bf a}(\btheta_0)+o_p(n^{-1}).
\label{eqn:mml:mle}
\end{align}
Moreover,
\begin{align}
\operatorname{Bias}(\hat{\btheta}_{\rm WF})
&=
\operatorname{Bias}(\hat{\btheta}_{\rm MLE})
+
K(\btheta_0)^{-1}{\bf a}(\btheta_0)
+
o(n^{-1})
\nonumber\\
&=
K(\btheta_0)^{-1}A(\btheta_0)\operatorname{vec}\!\bigl(K(\btheta_0)^{-1}\bigr)
+
K(\btheta_0)^{-1}{\bf a}(\btheta_0)
+
o(n^{-1}),
\end{align}
where $A(\btheta_0)$ is the Cox--Snell cumulant matrix in \eqref{eqn:cumulants}.
\end{theorem}
Assumption (C5) is a high-level remainder condition ensuring that the Taylor expansion of the Wallace--Freeman score equation is valid at the $n^{-1}$ scale and that expectations may be taken term-wise.

\begin{corollary}
\label{cor:wf87:jeffreys}
If $\pi(\btheta)\propto |I(\btheta)|^{1/2}$, then ${\bf a}(\btheta)\equiv 0$, so
\[
\hat{\btheta}_{\rm WF}-\hat{\btheta}_{\rm MLE}=o_p(n^{-1}),
\]
and the two estimators have the same $O(n^{-1})$ bias. 
If instead $\pi(\btheta)\propto |I(\btheta)|$, then
\[
{\bf a}(\btheta)=\frac12 \nabla_\theta \log |I(\btheta)|,
\]
and minimising the Wallace--Freeman codelength is equivalent to maximising Firth's penalised likelihood \cite{Firth93} in canonical exponential families. 
\end{corollary}
The proof of Corollary~\ref{cor:wf87:jeffreys} follows immediately from \eqref{eqn:wfbias:a}. 
The Wallace--Freeman estimator includes two important limiting cases. With Jeffreys prior, the \(O(n^{-1})\) displacement from the MLE vanishes. With the prior proportional to \(|I(\btheta)|\), the Wallace--Freeman penalty coincides with the Jeffreys-type penalty appearing in Firth's bias-reducing likelihood. 
\section{Weibull distribution example}
\label{sec:example}
We illustrate the general bias formula using the Weibull distribution with density
\begin{align}
    f(x \mid k,\lambda)
    =
    \frac{k}{\lambda}
    \left(\frac{x}{\lambda}\right)^{k-1}
    \exp\!\left\{-\left(\frac{x}{\lambda}\right)^k\right\},
    \qquad x>0,
\end{align}
where $\btheta=(k,\lambda)$, with $k>0$ the shape parameter and $\lambda>0$ the scale parameter. For observations $\bx=(x_1,\ldots,x_n)$, the full log-likelihood is
\begin{align}
    \ell(\btheta)
    =
    n(\log k-k\log\lambda)
    +(k-1)\sum_{i=1}^n \log x_i
    -
    \sum_{i=1}^n \left(\frac{x_i}{\lambda}\right)^k.
\end{align}
The expected per-sample Fisher information matrix and corresponding determinant are
\begin{align}
I(\btheta)
=
\begin{pmatrix}
\dfrac{6(\gamma-1)^2+\pi^2}{6k^2} & \dfrac{\gamma-1}{\lambda} \\[1.2ex]
\dfrac{\gamma-1}{\lambda} & \dfrac{k^2}{\lambda^2}
\end{pmatrix},
\qquad
|I(\btheta)|
=
\frac{\pi^2}{6\lambda^2},
\end{align}
with the inverse Fisher information matrix given by
\begin{align}
I(\btheta)^{-1}
=
\frac{6}{\pi^2}
\begin{pmatrix}
k^2 & -(\gamma-1)\lambda \\[0.8ex]
-(\gamma-1)\lambda &
\dfrac{\{(\gamma-1)^2+\pi^2/6\}\lambda^2}{k^2}
\end{pmatrix},
%
% \frac{1}{\pi^2}
% \begin{pmatrix}
% 6k^2 & -6(\gamma-1)\lambda \\[0.8ex]
% -6(\gamma-1)\lambda &
% \dfrac{\{6(\gamma-1)^2+\pi^2\}\lambda^2}{k^2}
% \end{pmatrix},
\end{align}
where $\gamma \approx 0.5772$ denotes the Euler--Mascheroni constant.
Substituting the corresponding cumulant matrix $A$ into the Cox--Snell formula \eqref{eqn:bias:MLE:ck} yields the first-order bias of the maximum likelihood estimator \cite{TanakaEtAl17}:
\begin{align}
\operatorname{Bias}(\hat{\btheta}_{\rm MLE})
=
\frac1n
\begin{pmatrix}
\dfrac{18k(\pi^2-2\zeta(3))}{\pi^4} \\[2ex]
\dfrac{\lambda\left\{72(\gamma-1)k\zeta(3)+6\pi^2\bigl(5k+\gamma(-4k+\gamma-2)+1\bigr)+\pi^4(1-2k)\right\}}{2\pi^4k^2}
\end{pmatrix}
+
O(n^{-2}),
\label{eq:weibull-mle-bias}
\end{align}
where $\zeta(3) \approx 1.2021$ denotes Ap\'ery's constant. For the Weibull model, substituting the Cox--Snell expansion for the maximum likelihood estimator and the first-order Wallace--Freeman correction from Theorem~\ref{thm:wf87:bias} yields the first-order bias formula. To evaluate the Wallace--Freeman bias, assume independent half-Cauchy priors,
\begin{align}
    \pi(k,\lambda)
    =
    \pi_k(k)\pi_\lambda(\lambda)
    =
    \frac{2}{\pi(1+k^2)}
    \frac{2}{\pi(1+\lambda^2)},
\end{align}
which are often used as weakly informative priors for positive and scale-type parameters \cite{Gelman06,PolsonScott12}. Then
\begin{align}
{\bf a}(\btheta)
=
-\nabla_\theta \operatorname{pen}(\btheta)
=
\begin{pmatrix}
-\dfrac{2k}{k^2+1} \\[1.5ex]
\dfrac{1-\lambda^2}{\lambda(\lambda^2+1)}
\end{pmatrix}.
\end{align}
Applying Theorem~\ref{thm:wf87:bias}, we obtain
\begin{align}
\operatorname{Bias}(\hat{\btheta}_{\rm WF})
=
\operatorname{Bias}(\hat{\btheta}_{\rm MLE})
+
\frac1n
\begin{pmatrix}
\dfrac{6}{\pi^2}
\left\{
\dfrac{(\gamma-1)(\lambda^2-1)}{\lambda^2+1}
-
\dfrac{2k^3}{k^2+1}
\right\}
\\[3ex]
\dfrac{\lambda}{\pi^2k^2}
\left\{
\dfrac{12(\gamma-1)k^3}{k^2+1}
-
\dfrac{\{6(\gamma-1)^2+\pi^2\}(\lambda^2-1)}{\lambda^2+1}
\right\}
\end{pmatrix}
+
o(n^{-1}).
\label{eq:weibull-mml-bias}
\end{align}
The first term in \eqref{eq:weibull-mml-bias} is the known Cox--Snell bias of the MLE. The second term is the new Wallace--Freeman contribution obtained from Theorem~\ref{thm:wf87:bias}. 

From \eqref{eq:weibull-mle-bias}, we see that the maximum likelihood estimate of the Weibull shape parameter is positively biased:
\begin{align}
\operatorname{Bias}(\hat{k}_{\rm MLE})
=
\frac{18k(\pi^2-2\zeta(3))}{n\pi^4}
+
O(n^{-2})
\approx
\frac{1.38\,k}{n}
+
O(n^{-2}),
\label{eq:weibull-shape-bias}
\end{align}
with the bias increasing linearly with $k$.
The additional Wallace--Freeman term in \eqref{eq:weibull-mml-bias} counteracts this upward shift through the penalty gradient ${\bf a}(\btheta)$, which depends on both the prior and the Fisher information contribution. For example, when $\lambda=1$, the ratio of the $O(n^{-1})$ MLE bias to the corresponding Wallace--Freeman bias for the shape parameter simplifies to
\begin{align}
R(k,\lambda=1)
=
\frac{
3(k^2+1)(\pi^2-2\zeta(3))
}{
\pi^2(k^2+3)-6(k^2+1)\zeta(3)
},
\end{align}
which is greater than one for all $k>0$. 
More generally, the leading correction term in \eqref{eq:weibull-mml-bias} indicates that the Wallace--Freeman estimator typically reduces the first-order upward bias of the Weibull shape estimator.

\begin{figure}[t]
\centering
\includegraphics[width=0.8\textwidth]{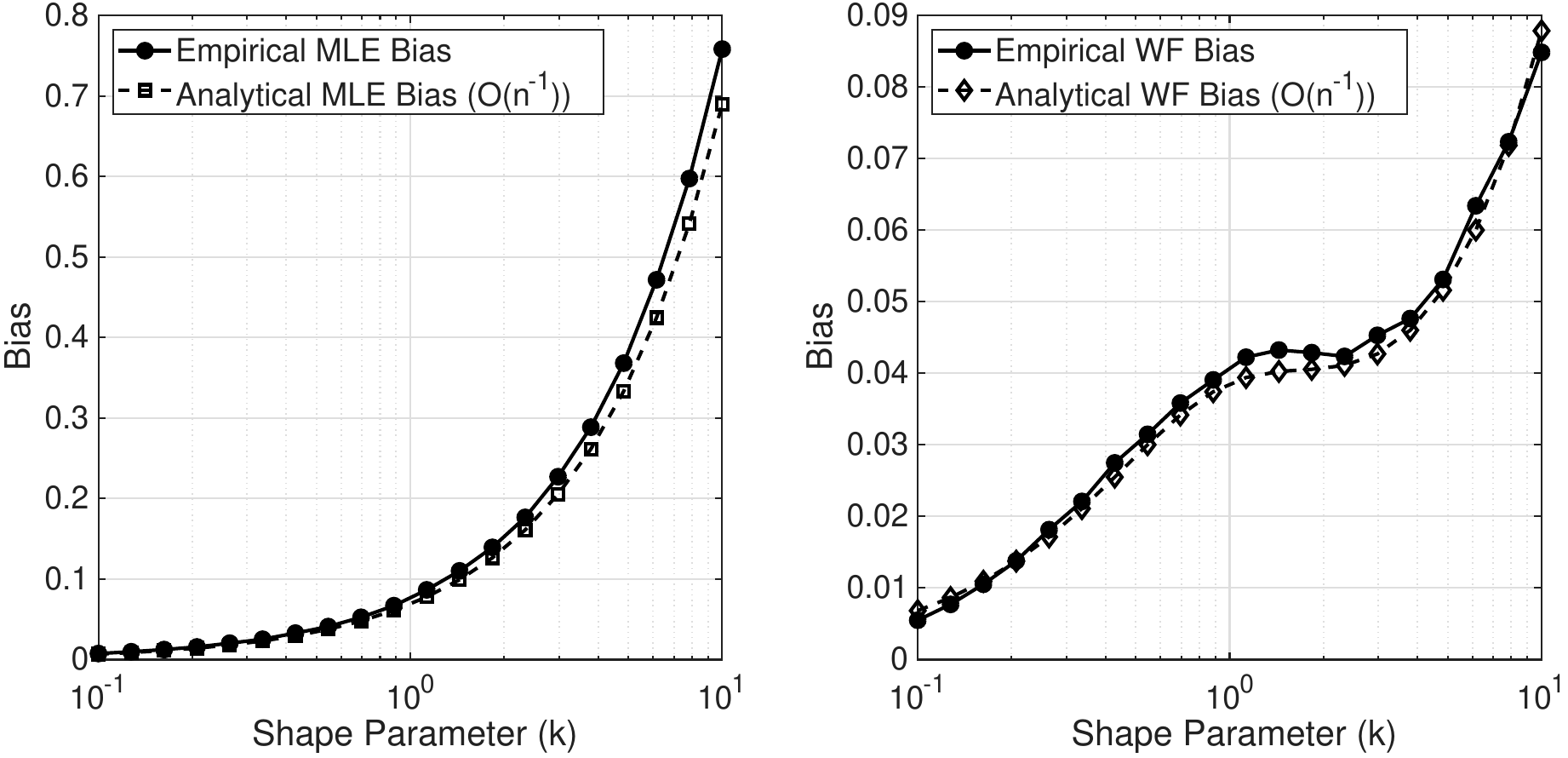}
\caption{Empirical and analytical \(O(n^{-1})\) biases for the Weibull shape parameter. The left panel compares the empirical and analytical maximum likelihood bias. The right panel compares the empirical and analytical Wallace--Freeman bias. The simulation uses \(n=20\), \(\lambda=1\), 20 logarithmically spaced values of \(k\in[0.1,10]\), and \(10^5\) Monte Carlo replications for each value of \(k\).}
\label{fig:weibull-analytical-empirical}
\end{figure}

\subsection{Simulation}
\label{sec:simulation}
We complement the analytic derivations with a brief Monte Carlo study of the first-order bias approximation for the Weibull distribution shape parameter. The simulation assesses whether the \(O(n^{-1})\) bias expressions derived in Section~\ref{sec:example} provide an accurate finite-sample description of the Wallace--Freeman and maximum likelihood estimators. 

Data were generated from the Weibull distribution with scale parameter fixed at \(\lambda=1\). The true shape parameter \(k\) was varied over 20 logarithmically spaced values in the interval \([0.1,10]\). For each value of \(k\), \(10^5\) independent samples of size \(n=20\) were generated. For each sample we computed the maximum likelihood estimator and the Wallace--Freeman estimator under the independent half-Cauchy prior used in Section~\ref{sec:example}. The Wallace--Freeman optimisation was initialised at the maximum likelihood estimate. The empirical bias was then estimated as
\[
\widehat{\operatorname{Bias}}(\hat k)
=
\frac{1}{B}\sum_{b=1}^B (\hat k^{(b)} - k),
\qquad B=10^5,
\]
where $\hat k$ denotes the Wallace--Freeman or the maximum likelihood estimate of $k$. The resulting empirical biases were compared with the first-order analytical bias formulae in \eqref{eq:weibull-mle-bias} and \eqref{eq:weibull-mml-bias}. The MATLAB code used to reproduce the simulation results and figures is available in the \href{https://github.com/EnesMakalic/Asymptotic-theory-of-the-WF-estimator}{GitHub repository}.

A comparison of the empirical and analytical \(O(n^{-1})\) biases for the maximum likelihood and Wallace--Freeman estimators of $k$ is shown in Figure~\ref{fig:weibull-analytical-empirical}. The agreement is close across the range of shape parameters considered. For the maximum likelihood estimator, the empirical bias follows the analytical curve almost exactly, with the bias increasing approximately linearly in \(k\). For the Wallace--Freeman estimator, the analytical approximation also tracks the empirical bias well, although some small discrepancies are visible for larger values of \(k\).

Figure~\ref{fig:weibull-mle-mml-bias}  compares the empirical biases of the two estimators. The maximum likelihood estimator exhibits the expected positive bias in the Weibull shape parameter, and this bias increases substantially as \(k\) grows. By contrast, the Wallace--Freeman estimator has a much smaller empirical bias across the same range. This behaviour is consistent with the analytical correction term in \eqref{eq:weibull-mml-bias}, which offsets the leading Cox--Snell bias of the maximum likelihood estimator when \(\lambda=1\). The simulation therefore supports the first-order expansion and illustrates that the Wallace--Freeman penalty can lead to a substantial reduction in shape-parameter bias in this setting.

\begin{figure}[t]
\centering
\includegraphics[width=0.6\textwidth]{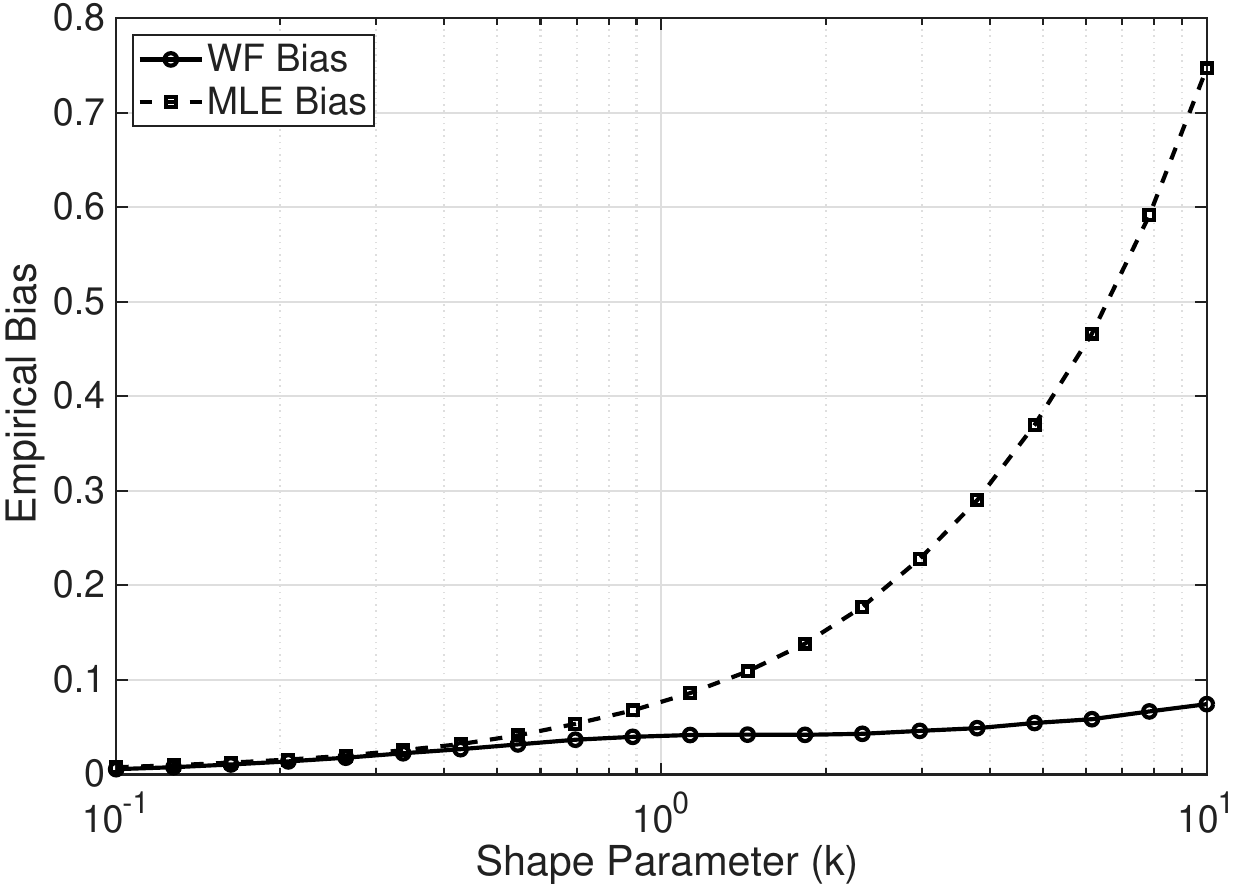}
\caption{Empirical bias of the Wallace--Freeman and maximum likelihood estimators of the Weibull shape parameter. The maximum likelihood estimator has positive bias that increases rapidly with \(k\), whereas the Wallace--Freeman estimator has substantially smaller empirical bias over the range considered. The simulation uses \(n=20\), \(\lambda=1\), and \(10^5\) Monte Carlo replications for each value of \(k\).}
\label{fig:weibull-mle-mml-bias}
\end{figure}
\section{Discussion and Future Work}
\label{sec:discussion}

We have shown that the Wallace--Freeman estimator can be analysed as a penalised likelihood estimator with penalty weight \(n^{-1}\). This representation yields the usual first-order likelihood asymptotics while also identifying the \(O(n^{-1})\) displacement from the maximum likelihood estimator. Combining this displacement with the Cox--Snell expansion gives a simple first-order bias formula for the Wallace--Freeman estimator.

The results clarify the role of the Wallace--Freeman penalty in regular parametric inference. The Fisher information term ensures invariance under smooth reparameterisation, while the prior controls the direction and magnitude of the first-order bias modification. Special choices of the prior recover important limiting cases, including agreement with the MLE to order \(n^{-1}\) under Jeffreys prior and a connection with Firth's penalised likelihood when \(\pi(\btheta)\propto |I(\btheta)|\).

The present results describe the Wallace--Freeman estimator at the \(n^{-1/2}\) and \(n^{-1}\) scales, but its behaviour beyond this first-order setting remains open. Higher-order expansions would clarify how the penalty affects terms beyond the leading bias, including possible second-order effects on variance and efficiency. Under model misspecification, the penalised M-estimation framework already allows \({\bf A}\neq{\bf B}\), suggesting an extension to sandwich asymptotics where the Fisher information penalty need not match the curvature of the limiting risk. Further work could also address nonregular settings, such as boundary parameters or singular information matrices, and models in which \(J_n(\btheta)\) grows at a nonstandard rate and the effective penalty order changes.
\section*{Acknowledgements}
Generative AI tools (Claude Opus 4.6; Microsoft Copilot GPT-5.4) were used in the preparation of this manuscript for generation and exploration of ideas as well as language and MATLAB source code refactoring. The authors reviewed and are responsible for all content, including all mathematical derivations, citations, and conclusions.
\appendix
\section{Proof of Theorem~\ref{thm:asmml-expected-simplified}}
\label{app:asmml-expected-proof}

\begin{proof}[Proof of Theorem~\ref{thm:asmml-expected-simplified}]
For $\bx\in P_{n,j}$, the code length is
\[
\Lambda_{n,j}(\bx) = -\log q_{n,j}-\log p_n(\bx\mid\btheta_{n,j}).
\]
Adding and subtracting the local continuous fitted likelihood gives
\[
% -\log q_{n,j}
% -\log p_n(\bx\mid\btheta_{n,j})
\Lambda_{n,j}(\bx) 
=
-\log p_n\{\bx\mid T_n(\bx)\}
+
\left[
-\log q_{n,j}
-\log p_n(\bx\mid\btheta_{n,j})
+
\log p_n\{\bx\mid T_n(\bx)\}
\right].
\]
Taking expectation with respect to $r_n$, we have
\[
\mathcal I_n^{\mathrm{ASMML}}
=
\mathbb E_{r_n}
\left[
-\log p_n\{\bX\mid T_n(\bX)\}
\right]
+
\sum_j q_{n,j}
\left\{
-\log q_{n,j}
+
\Delta_{n,j}
\right\}.
\]
We first expand the assertion term. By the localisation assumption,
\[
q_{n,j}
=
\Pi(W_{n,j})\exp(\alpha_{n,j}),
\qquad
\sum_j q_{n,j}|\alpha_{n,j}|\to0.
\]
Hence
\[
-\log q_{n,j}
=
-\log \Pi(W_{n,j})
-
\alpha_{n,j},
\]
and consequently
\[
\sum_j q_{n,j}(-\log q_{n,j})
=
\sum_j q_{n,j}(-\log \Pi(W_{n,j}))
+
o(1).
\]
By the assumed Fisher-normalised cell representation, we have
\[
W_{n,j}
=
\btheta_{n,j}
+
J_n(\btheta_{n,j})^{-1/2}C_{n,j}.
\]
Using the change of variables
\[
\bvartheta
=
\btheta_{n,j}
+
J_n(\btheta_{n,j})^{-1/2}\bu,
\qquad
\bu\in C_{n,j},
\]
we obtain
\[
\Pi(W_{n,j})
=
|J_n(\btheta_{n,j})|^{-1/2}
\int_{C_{n,j}}
\pi\!\left(
\btheta_{n,j}
+
J_n(\btheta_{n,j})^{-1/2}\bu
\right)
d\bu
.
\]
The sets $C_{n,j}$ have uniformly bounded diameter, while
\[
\inf_{\btheta\in K}\lambda_{\min}\{J_n(\btheta)\}\to\infty.
\]
Thus
\[
\sup_{\btheta\in K}
\left\|J_n(\btheta)^{-1/2}\right\|
\to0,
\]
and the corresponding cells shrink uniformly in ordinary parameter space. Since $\pi$ is continuous and strictly positive on the compact set $K$, it is uniformly continuous and bounded away from zero on $K$. Therefore the prior is approximately locally constant over these shrinking cells, and
\[
\Pi(W_{n,j})
=
\pi(\btheta_{n,j})
|J_n(\btheta_{n,j})|^{-1/2}
V_{n,j}
\{1+o(1)\}, \qquad V_{n,j} = {\rm Vol}(C_{n,j}),
\]
with the induced logarithmic error negligible after averaging with respect to $q_{n,j}$. Hence
\[
\sum_j q_{n,j}(-\log q_{n,j})
=
\sum_j q_{n,j}
\left[
-\log\pi(\btheta_{n,j})
+
\frac12\log|J_n(\btheta_{n,j})|
-
\log V_{n,j}
\right]
+
o(1).
\]
We next expand the second-part regret $\Delta_{n,j}$. By the assumed KL representation,
\[
\Delta_{n,j}
=
\frac{1}{\Pi(W_{n,j})}
\int_{W_{n,j}}
D_{\mathrm{KL}}\!\left(
P_{\bvartheta}^{(n)}
\Vert
P_{\btheta_{n,j}}^{(n)}
\right)
\pi(\bvartheta)\,d\bvartheta
+
\varepsilon_{n,j},
\]
where
\[
\sum_j q_{n,j}|\varepsilon_{n,j}|\to0.
\]
Therefore
\[
\sum_j q_{n,j}\Delta_{n,j}
=
\sum_j q_{n,j}
\frac{1}{\Pi(W_{n,j})}
\int_{W_{n,j}}
D_{\mathrm{KL}}\!\left(
P_{\bvartheta}^{(n)}
\Vert
P_{\btheta_{n,j}}^{(n)}
\right)
\pi(\bvartheta)\,d\bvartheta
+
o(1).
\]

For $\bvartheta\in W_{n,j}$, write
\[
\bu
=
J_n(\btheta_{n,j})^{1/2}
(\bvartheta-\btheta_{n,j}).
\]
Then $\bu\in C_{n,j}$. Since the sets $C_{n,j}$ have uniformly bounded diameter, all such $\bu$ lie in a fixed bounded subset of $\mathbb R^d$. Assumption~\ref{ass:simplified:kl} therefore gives, uniformly over the cells,
\[
D_{\mathrm{KL}}\!\left(
P_{\bvartheta}^{(n)}
\Vert
P_{\btheta_{n,j}}^{(n)}
\right)
=
\frac12\|\bu\|^2+o(1).
\]
Using again the change of variables
\[
\bvartheta
=
\btheta_{n,j}
+
J_n(\btheta_{n,j})^{-1/2}\bu,
\]
and using the fact that $\pi$ is approximately  constant over the shrinking cells, the prior-weighted cell average is asymptotically the uniform average over $C_{n,j}$. Thus
\[
\sum_j q_{n,j}\Delta_{n,j}
=
\sum_j q_{n,j}
\frac{1}{2V_{n,j}}
\int_{C_{n,j}}\|\bu\|^2\,d\bu
+
o(1).
\]

Combining the assertion and regret expansions gives
\[
\mathcal I_n^{\mathrm{ASMML}}
=
\mathbb E_{r_n}
\left[
-\log p_n\{\bX\mid T_n(\bX)\}
\right]
+
\sum_j q_{n,j}
\left[
-\log\pi(\btheta_{n,j})
+
\frac12\log|J_n(\btheta_{n,j})|
\right]
\]
\[
\qquad
+
\sum_j q_{n,j}
\left\{
-\log V_{n,j}
+
\frac{1}{2V_{n,j}}
\int_{C_{n,j}}\|\bu\|^2\,d\bu
\right\}
+
o(1).
\]
By the assumed asymptotic high-rate quantisation optimality,
\[
\sum_j q_{n,j}
\left\{
-\log V_{n,j}
+
\frac{1}{2V_{n,j}}
\int_{C_{n,j}}\|\bu\|^2\,d\bu
\right\}
=
\frac d2\log\kappa_d+\frac d2+o(1).
\]
By the prior-weighted empirical convergence assumption,
\[
\sum_j q_{n,j}
\left[
-\log\pi(\btheta_{n,j})
+
\frac12\log|J_n(\btheta_{n,j})|
\right]
=
\int_K
\left[
-\log\pi(\btheta)
+
\frac12\log|J_n(\btheta)|
\right]
\pi(\btheta)\,d\btheta
+
o(1).
\]
Since
\[
h(\pi)
=
-\int_K \pi(\btheta)\log\pi(\btheta)\,d\btheta,
\]
this becomes
\[
\sum_j q_{n,j}
\left[
-\log\pi(\btheta_{n,j})
+
\frac12\log|J_n(\btheta_{n,j})|
\right]
=
h(\pi)
+
\frac12
\int_K\pi(\btheta)\log|J_n(\btheta)|\,d\btheta
+
o(1).
\]
Therefore
\[
\mathcal I_n^{\mathrm{ASMML}}
=
\mathbb E_{r_n}
\left[
-\log p_n\{\bX\mid T_n(\bX)\}
\right]
+
h(\pi)
+
\frac12\int_K\pi(\btheta)\log|J_n(\btheta)|\,d\btheta
+
\frac d2\log\kappa_d
+
\frac d2
+
o(1).
\]

Finally, if $J_n(\btheta)=n I(\btheta)$, then
\[
\log|J_n(\btheta)|
=
\log|nI(\btheta)|
=
d\log n+\log|I(\btheta)|,
\]
completing the proof.
\end{proof}

\section{Proof of Theorem~\ref{thm:wf87:penm}}
\label{app:penm-proof}

\begin{proof}[Proof of Theorem~\ref{thm:wf87:penm}]
By Assumption (B6), the criterion $Q_n$ is continuous and either $\Theta$ is compact, or $\Theta$ is closed and $Q_n$ is coercive with probability tending to one. Hence, by the Weierstrass theorem, a global minimiser exists with probability tending to one.

For consistency, Assumptions (B1) and (B3) imply uniform convergence of the empirical risk to $Q$ on compact sets, and $Q$ has a unique minimiser at $\btheta_0$. Since $\lambda_n\to0$ and the penalty is continuous and locally bounded near $\btheta_0$, the penalised criterion has the same local limiting objective. The argmin theorem therefore gives 
\[
\hat\btheta_n\xrightarrow{p}\btheta_0.
\]

Because $\btheta_0\in\operatorname{int}(\Theta)$, consistency implies that $\hat\btheta_n$ lies in an interior neighbourhood of $\btheta_0$ with probability tending to one. On this event the stationarity equation holds:
\[
0
=
\Psi_n(\hat\btheta_n)
=
\frac1n\sum_{i=1}^n\psi(X_i,\hat\btheta_n)
+
\lambda_n\nabla_\theta\operatorname{pen}(\hat\btheta_n).
\]
A first-order Taylor expansion around $\btheta_0$ gives
\[
0
=
\Psi_n(\btheta_0)
+
\nabla_\theta\Psi_n(\tilde\btheta_n)
(\hat\btheta_n-\btheta_0),
\]
where $\tilde\btheta_n$ lies between $\hat\btheta_n$ and $\btheta_0$.
By Assumptions (B5), (B8), and consistency,
\[
\nabla_\theta\Psi_n(\tilde\btheta_n)
=
\frac1n\sum_{i=1}^n\nabla_\theta\psi(X_i,\tilde\btheta_n)
+
\lambda_n\nabla_\theta^2\operatorname{pen}(\tilde\btheta_n)
\xrightarrow{p}
{\bf A}.
\]
Hence
\[
\hat\btheta_n-\btheta_0
=
-{\bf A}^{-1}\Psi_n(\btheta_0)
+
o_p(\|\Psi_n(\btheta_0)\|).
\]
Since
\[
\Psi_n(\btheta_0)
=
\frac1n\sum_{i=1}^n\psi(X_i,\btheta_0)
+
\lambda_n\nabla_\theta\operatorname{pen}(\btheta_0),
\]
we obtain
\[
\hat\btheta_n-\btheta_0
=
-{\bf A}^{-1}
\left[
\frac1n\sum_{i=1}^n\psi(X_i,\btheta_0)
+
\lambda_n\nabla_\theta\operatorname{pen}(\btheta_0)
\right]
+
o_p(n^{-1/2}+\lambda_n).
\]
When \(\lambda_n=n^{-1}\), the deterministic penalty term is of smaller order than \(n^{-1/2}\), so the expansion simplifies to
\[
\hat\btheta_n-\btheta_0
=
-\mathbf{A}^{-1}
\frac1n\sum_{i=1}^n\psi(X_i,\btheta_0)
+
o_p(n^{-1/2}),
\]
which proves the stated rate and linear expansion.

Finally, Assumption (B7) and the condition
$\sqrt n\,\lambda_n\to0$ imply
\[
\sqrt n(\hat\btheta_n-\btheta_0)
=
-{\bf A}^{-1}
\frac1{\sqrt n}\sum_{i=1}^n\psi(X_i,\btheta_0)
+
o_p(1),
\]
and hence
\[
\sqrt n(\hat\btheta_n-\btheta_0)
\xrightarrow{d}
N(0,{\bf A}^{-1}{\bf B}{\bf A}^{-1}).
\]
In the correctly specified likelihood case,
${\bf A}={\bf B}=I(\btheta_0)$, giving the usual inverse Fisher
information covariance.
\end{proof}

\section{Proof of Theorem~\ref{thm:wf87:bias}}
\label{app:bias-proof}

\begin{proof}[Proof of Theorem~\ref{thm:wf87:bias}]
The Wallace--Freeman estimator satisfies
\[
\nabla\ell(\hat\btheta_{\rm WF})
+
{\bf a}(\hat\btheta_{\rm WF})
=
0,
\qquad
{\bf a}(\btheta)
=
\nabla\log\pi(\btheta)
-
\frac12\nabla\log |I(\btheta)|.
\]
Let
\[
\Delta
=
\hat\btheta_{\rm WF}-\hat\btheta_{\rm MLE}.
\]
Since
\[
\nabla\ell(\hat\btheta_{\rm MLE})=0,
\]
a Taylor expansion of the score about $\hat\btheta_{\rm MLE}$ gives
\[
\nabla\ell(\hat\btheta_{\rm WF})
=
\nabla^2\ell(\hat\btheta_{\rm MLE})\Delta
+
R_n,
\]
where the remainder is negligible at the scale required by Assumption
(C5). Thus
\[
\nabla^2\ell(\hat\btheta_{\rm MLE})\Delta
+
{\bf a}(\hat\btheta_{\rm WF})
+
R_n
=
0.
\]
The observed information satisfies
\[
-\nabla^2\ell(\hat\btheta_{\rm MLE})
=
K(\btheta_0)+O_p(n^{1/2}),
\]
with $K(\btheta_0)=nI(\btheta_0)$. Since $K(\btheta_0)^{-1}=O(n^{-1})$ and ${\bf a}(\btheta_0)=O(1)$, we have
\[
\Delta=O_p(n^{-1}).
\]
By differentiability of ${\bf a}$ and root-$n$ consistency of
$\hat\btheta_{\rm MLE}$,
\[
{\bf a}(\hat\btheta_{\rm WF})
=
{\bf a}(\btheta_0)+O_p(n^{-1/2}).
\]
Premultiplication by the inverse observed information makes this
$O_p(n^{-1/2})$ perturbation contribute only $O_p(n^{-3/2})$, which is $o_p(n^{-1})$. Therefore
\[
\Delta
=
\left[-\nabla^2\ell(\hat\btheta_{\rm MLE})\right]^{-1}
{\bf a}(\btheta_0)
+
o_p(n^{-1}).
\]
Moreover,
\[
\left[-\nabla^2\ell(\hat\btheta_{\rm MLE})\right]^{-1}
=
K(\btheta_0)^{-1}
+
o_p(n^{-1}),
\]
because
\[
-\nabla^2\ell(\hat\btheta_{\rm MLE})
=
K(\btheta_0)+O_p(n^{1/2})
\]
and $K(\btheta_0)^{-1}=O(n^{-1})$. Hence
\[
\hat\btheta_{\rm WF}-\hat\btheta_{\rm MLE}
=
K(\btheta_0)^{-1}{\bf a}(\btheta_0)
+
o_p(n^{-1})
=
\frac1n I(\btheta_0)^{-1}{\bf a}(\btheta_0)
+
o_p(n^{-1}),
\]
which proves \eqref{eqn:mml:mle}.

For the bias expansion, decompose
\[
\mathbb E_{\btheta_0}(\hat\btheta_{\rm WF}-\btheta_0)
=
\mathbb E_{\btheta_0}(\hat\btheta_{\rm MLE}-\btheta_0)
+
\mathbb E_{\btheta_0}(\hat\btheta_{\rm WF}-\hat\btheta_{\rm MLE}).
\]
By the preceding expansion and the uniform integrability condition in
Assumption (C5),
\[
\mathbb E_{\btheta_0}(\hat\btheta_{\rm WF}-\hat\btheta_{\rm MLE})
=
K(\btheta_0)^{-1}{\bf a}(\btheta_0)
+
o(n^{-1}).
\]
Combining this with the Cox--Snell expansion
\[
\operatorname{Bias}(\hat\btheta_{\rm MLE})
=
K(\btheta_0)^{-1}
A(\btheta_0)
\operatorname{vec}\{K(\btheta_0)^{-1}\}
+
O(n^{-2})
\]
gives
\[
\operatorname{Bias}(\hat\btheta_{\rm WF})
=
K(\btheta_0)^{-1}
A(\btheta_0)
\operatorname{vec}\{K(\btheta_0)^{-1}\}
+
K(\btheta_0)^{-1}{\bf a}(\btheta_0)
+
o(n^{-1}).
\]
This proves the theorem.
\end{proof}

\bibliographystyle{unsrt}
\bibliography{bibliography}  

@Book{ConwaySloane98,
  Title                    = {Sphere Packing, Lattices and Groups},
  Author                   = {J.~H. Conway and N.~J.~A. Sloane},
  Publisher                = {Springer-Verlag},
  Year                     = {1998},
  Edition                  = {Third},
  Month                    = {December},
  Pages                    = {703}
}

@Article{FarrWallace02,
  author    = {Graham E. Farr and Chris S. Wallace},
  journal   = {Computer Journal},
  title     = {The complexity of Strict Minimum Message Length inference},
  year      = {2002},
  number    = {3},
  pages     = {285--292},
  volume    = {45},
  doi       = {10.1093/comjnl/45.3.285},
}

@Article{Firth93,
  author    = {David Firth},
  title     = {Bias Reduction of Maximum Likelihood Estimates},
  year      = {1993},
  volume    = {80},
  number    = {1},
  pages     = {27--38},
  journal   = {Biometrika},
}

@Article{Gelman06,
  Title                    = {Prior distributions for variance parameters in hierarchical models},
  Author                   = {Andrew Gelman},
  Journal                  = {Bayesian Analysis},
  Year                     = {2006},
  Number                   = {3},
  Pages                    = {515--533},
  Volume                   = {1}
}

@Article{KullbackLeibler51,
  author    = {S. Kullback and R.~A. Leibler},
  journal   = {The Annals of Mathematical Statistics},
  title     = {On information and sufficiency},
  year      = {1951},
  month     = {March},
  number    = {1},
  pages     = {79--86},
  volume    = {22},
  doi       = {10.1214/aoms/1177729694},
}

@Article{Schwarz78,
  author    = {Gideon Schwarz},
  title     = {Estimating the dimension of a model},
  year      = {1978},
  volume    = {6},
  number    = {2},
  pages     = {461--464},
  journal   = {The Annals of Statistics},
}

@Article{WallaceBoulton75,
  author    = {Chris S. Wallace and David~M. Boulton},
  journal   = {Classification Society Bulletin},
  title     = {An invariant {B}ayes method for point estimation},
  year      = {1975},
  number    = {3},
  pages     = {11--34},
  volume    = {3},
}

@Book{Wallace05,
  author    = {Chris~S. Wallace},
  publisher = {Springer},
  title     = {Statistical and inductive inference by minimum message length},
  year      = {2005},
  edition   = {First},
  series    = {Information Science and Statistics},
  doi       = {10.1007/0-387-27656-4},
}

@Article{WallaceBoulton68,
  author    = {Chris S. Wallace and David~M. Boulton},
  journal   = {Computer Journal},
  title     = {An information measure for classification},
  year      = {1968},
  month     = {August},
  number    = {2},
  pages     = {185--194},
  volume    = {11},
  doi       = {10.1093/comjnl/11.2.185},
}

@Article{WallaceDowe99a,
  author    = {Chris S. Wallace and David L. Dowe},
  journal   = {Computer Journal},
  title     = {Refinements of {MDL} and {MML} Coding},
  year      = {1999},
  number    = {4},
  pages     = {330--337},
  volume    = {42},
  doi       = {10.1093/comjnl/42.4.330},
}

@Article{WallaceDowe99c,
  author    = {Chris S. Wallace and David L. Dowe},
  journal   = {Computer Journal},
  title     = {Minimum Message Length and {K}olmogorov Complexity},
  year      = {1999},
  number    = {4},
  pages     = {270--283},
  volume    = {42},
  doi       = {10.1093/comjnl/42.4.270},
}

@Article{WallaceFreeman87,
  author    = {Chris~S. Wallace and Peter~R. Freeman},
  journal   = {Journal of the Royal Statistical Society (Series B)},
  title     = {Estimation and inference by compact coding},
  year      = {1987},
  number    = {3},
  pages     = {240--252},
  volume    = {49},
  doi       = {10.1111/j.2517-6161.1987.tb01695.x},
}

@Article{PolsonScott12,
  author    = {Nicholas G. Polson and James G. Scott},
  journal   = {Bayesian Analysis},
  title     = {On the Half-{C}auchy Prior for a Global Scale Parameter},
  year      = {2012},
  number    = {4},
  volume    = {7},
  doi       = {10.1214/12-BA730},
  publisher = {Institute of Mathematical Statistics},
}

@Article{CoxSnell68,
  author    = {D. R. Cox and E. J. Snell},
  journal   = {Journal of the Royal Statistical Society: Series B (Methodological)},
  title     = {A General Definition of Residuals},
  year      = {1968},
  number    = {2},
  pages     = {248--265},
  volume    = {30},
  doi       = {10.1111/j.2517-6161.1968.tb00724.x},
  publisher = {Wiley},
}

@Article{CordeiroKlein94,
  author    = {Gauss M. Cordeiro and Ruben Klein},
  journal   = {Statistics {\&} Probability Letters},
  title     = {Bias correction in {ARMA} models},
  year      = {1994},
  number    = {3},
  pages     = {169--176},
  volume    = {19},
  doi       = {10.1016/0167-7152(94)90100-7},
  publisher = {Elsevier {BV}},
}

@Article{TanakaEtAl17,
  author       = {Hidekazu Tanaka and Nabendu Pal and Wooi K. Lim},
  date         = {2017-04},
  journal      = {Journal of Statistical Theory and Practice},
  title        = {On improved estimation under {W}eibull model},
  year         = {2017},
  doi          = {10.1080/15598608.2017.1305921},
  number       = {1},
  pages        = {48--65},
  volume       = {12},
  journaltitle = {Journal of Statistical Theory and Practice},
  publisher    = {Springer Science and Business Media {LLC}},
}

@Misc{Dowty13,
  author       = {Dowty, James G.},
  howpublished = {arXiv:1302.0581},
  title        = {{SMML} estimators for exponential families with continuous sufficient statistics},
  year         = {2013},
  copyright    = {arXiv.org perpetual, non-exclusive license},
  doi          = {10.48550/arXiv.1302.0581},
  publisher    = {arXiv},
}

@Article{MakalicSchmidt26b,
  author    = {Makalic, Enes and Schmidt, Daniel F.},
  journal   = {Entropy},
  title     = {Information Geometry and Asymptotic Theory for {SMML} Estimators},
  year      = {2026},
  issn      = {1099-4300},
  month     = June,
  number    = {6},
  pages     = {713},
  volume    = {28},
  doi       = {10.3390/e28060713},
  publisher = {MDPI AG},
}

\end{document}